\documentclass[psamsfonts]{amsart}
\usepackage{amssymb,amsfonts}
\usepackage{latexsym}
\usepackage{enumerate}
\usepackage{mathrsfs}
\usepackage{url}

\newtheorem{thm}{Theorem}[section]
\newtheorem{cor}[thm]{Corollary}
\newtheorem{prop}[thm]{Proposition}

\newtheorem{quest}[thm]{Question}
\newtheorem{conj}[thm]{Conjecture}

\theoremstyle{definition}
\newtheorem{defn}[thm]{Definition}

\theoremstyle{remark}
\newtheorem{rem}[thm]{Remark}

\numberwithin{equation}{section}

\title{Internal Structure of Addition Chains: Well-Ordering}
\author{Harry Altman}
\date{September 1, 2014}

\begin{document}

\newcommand{\cpx}[1]{\|#1\|}
\newcommand{\dft}{\delta}
\newcommand{\dftl}{\dft^\acl}
\newcommand{\dftA}{\dft^{A}}
\newcommand{\acl}{\ell}
\newcommand{\st}{{st}}
\newcommand{\posf}{f}
\newcommand{\Cs}{C_s}
\newcommand{\bc}{q}

\newcommand{\N}{{\mathbb N}}
\newcommand{\R}{{\mathbb R}}
\newcommand{\Z}{{\mathbb Z}}
\newcommand{\Q}{{\mathbb Q}}
\newcommand{\sD}{{\mathscr{D}}}

\newcommand{\floor}[1]{{\lfloor #1 \rfloor}}
\newcommand{\ceil}[1]{{\lceil #1 \ceil}}

\begin{abstract}
An \emph{addition chain} for $n$ is defined to be a sequence
$(a_0,a_1,\ldots,a_r)$ such that $a_0=1$, $a_r=n$, and, for any $1\le k\le r$, there
exist $0\le i, j<k$ such that $a_k = a_i + a_j$; the number $r$ is called the
length of the addition chain.  The shortest length among addition chains for
$n$, called the addition chain length of $n$, is denoted $\acl(n)$.  The number
$\acl(n)$ is always at least $\log_2 n$; in this paper we consider the
difference $\dftl(n):=\acl(n)-\log_2 n$, which we call the addition chain
defect.  First we use this notion to show that for any $n$, there exists $K$
such that for any $k\ge K$, we have $\acl(2^k n)=\acl(2^K n)+(k-K)$.  The main
result is
that the set of values of $\dftl$ is a well-ordered subset of $[0,\infty)$, with
order type $\omega^\omega$.  The results obtained here are analogous to the
results for integer complexity obtained in \cite{cpxwo} and \cite{paper1}.  We
also prove similar well-ordering results for restricted forms of addition chain
length, such as star chain length and Hansen chain length.
\end{abstract}

\maketitle

\section{Introduction}
\label{intro}

An \emph{addition chain} for $n$ is defined to be a sequence
$(a_0,a_1,\ldots,a_r)$ such that $a_0=1$, $a_r=n$, and, for any $1\le k\le r$, there
exist $0\le i, j<k$ such that $a_k = a_i + a_j$; the number $r$ is called the
length of the addition chain.  The shortest length among addition chains for
$n$, called the \emph{addition chain length} of $n$, is denoted $\acl(n)$.
Addition chains were introduced in 1894 by H. Dellac \cite{Dellac} and
reintroduced in 1937 by A.~Scholz \cite{aufgaben},
who raised a series of questions about them. 
They have  been much studied  in the context of computation of
powers, since an addition chain for $n$ of length $r$ allows one to compute
$x^n$ from $x$ using $r$ multiplications.  Extensive surveys on the topic can be
found in Knuth \cite[Section 4.6.3]{TAOCP2} and Subbarao \cite{subreview}.

Addition chain length is approximately logarithmic; it satisfies the bounds
\[ 
\log_2 n \le \acl(n) \le \floor{\log_2 n}+\nu_2(n)-1, 
\]
in which  $\nu_2(n)$ counts the number of $1$'s in the binary expansion of $n$.
A.~Brauer \cite{Brauer} proved in 1939 that $\acl(n)\sim \log_2 n$.

The addition chain length function $\acl(n)$ seems complicated and hard to
compute.  An outstanding open problem about it is the {\em Scholz-Brauer
conjecture} \cite[Question 3]{aufgaben}, which asserts that 
\[
\acl( 2^n -1) \le n+ \acl(n) -1.
\]
To investigate it Brauer \cite{Brauer} introduced a restricted type of addition
chain called a {\em star chain}, and later authors introduced other restricted
types of addition chains, such as {\em Hansen chains}, discussed  in Section
\ref{variations}.  Later  Knuth \cite{TAOCP2} introduced the quantity
$s(n):=\acl(n)-\floor{\log_2 n}$, which he called the number of \emph{small
steps} of $n$.  This notion was subsequently used by other authors
\cite{sub1962, stolarsky, thurber} investigating the general behavior of
$\acl(n)$ and the Scholz-Brauer conjecture.  The Scholz-Brauer conjecture has
been verified to hold for $n<5784689$, by computations of Clift \cite{clift}.

In this paper we introduce and study an invariant of addition chain length
related to small steps, where instead of rounding we subtract off the exact
logarithm $\log_2 n$.  
\begin{defn}
The \emph{addition chain defect} $\dftl(n)$ of $n$ is
 $$
 \dftl(n) := \acl(n) - \log_2 n.
 $$
\end{defn}

This quantity is related to the number of small steps of $n$ by the equation
\[
s(n)=\lceil\dftl(n)\rceil.
\]
The lower bound result above shows that
\[ 
\dftl(n) \ge 0,
\]
with equality holding for $n = 2^k$ for $k \ge 0$.  In a sense, $\dftl(n)$
encodes the ``hard part'' of computing $\acl(n)$; $\log_2 n$ is an
easy-to-compute approximation to $\acl(n)$, and $\dftl(n)$ is the extra little
bit that is not so easy to compute.  The object of this paper is to show that
the addition chain defect encodes a subtle structural regularity of the addition
chain length function.

\subsection{Main results} \label{sec10}

The main results of the paper concern the structure of the set of all addition
chain defect values.

\begin{defn}
We define $\mathscr{D}^\acl$ to be the set of all addition chain defect values:
\[ \mathscr{D}^\acl = \{ \dftl(n) : n\in \N \}. \]
\end{defn}

The main result of this paper is the following well-ordering  theorem.

\begin{thm} \label{frontpagethmspecial}
{\em ($\acl$-defect well-ordering theorem)}
The set $\mathscr{D}^\acl$ is a well-ordered subset of $\R$, of order type
$\omega^\omega$.
\end{thm}

This theorem may at first appear to come out of nowhere, but we will discuss why
is is true in Section~\ref{methods}.

A second result is related to the determination of the set of integers having a
given value $\alpha$ of the addition chain defect. We will show that If
$\dftl(n_1)= \dftl(n_2)=\alpha$ with $n_1 \ne n_2$ then it is necessary (but not
always sufficient)  that $n_1 = 2^k n_2$ for some (positive or negative) integer
$k$.

It is always the case that $\acl(2n)\le \acl(n)+1$, and the equality $\acl(2n)=
\acl(n)+1$ corresponds to $\dftl(2n) = \dftl(n)$.  One might hope that we always
have $\acl(2n)=\acl(n)+1$, but this is not the case; sometimes $\dftl(2n)<
\dftl(n)$.  In fact, infinitely many counterexamples are known (Thurber
\cite{thurber}).  However infinitely many $n$ have this
property, which is part of a stabilization phenomenon.

\begin{defn}
A number $m$ is called \emph{$\acl$-stable} if 
\[
\acl(2^k m)=\acl(m)+k, \quad \mbox{for all} \quad k \ge 0.
\] 
Otherwise it is called \emph{$\acl$-unstable}.
\end{defn}

Using the defect, we will prove:

\begin{thm}
\label{power-of-2-special}
{\em ($\acl$-stability theorem)} 
We have:
\begin{enumerate}
\item If $\alpha$ is a value of $\dftl$, and
\[
S(\alpha) := \{m:~~\dftl(m) = \alpha\}
\]
then there is a unique integer $n$ such that $S(\alpha)$ has either the form
$\{ n\cdot 2^k: 0 \le k \le K\}$ for some finite $K$ or else the form $\{ n\cdot
2^k: k\ge 0\}$. The integer $n$ will be  called the {\em leader} of
$S(\alpha)$.
\item The set $S(\alpha)$ is infinite if and only if $\alpha$ is the smallest
defect occurring among all defects $\dftl(2^k n)$ for $k\ge0$, where $n$ is the
leader of $S(\alpha)$.
\item For a fixed odd integer $n$, the sequence $\{ \dftl(n \cdot 2^k): \, k \ge
0\}$ is non-increasing.  This sequence takes on finitely many values, all
differing by integers, culminating in a smallest value $\alpha$ such that if
$\dftl(m)=\alpha$ and
$k\ge 0$, then
\[ \acl(m\cdot 2^k) = \acl(m) + k.\]
\end{enumerate}
\end{thm}

That is to say, while doubling a number $n$ may not increase its addition chain
length by precisely $1$, if one starts with a fixed $n$ and begins doubling,
eventually one will reach a point where the length goes up by $1$ each time.
This result is easy to prove and is established in  Section~\ref{secdft}.

We use Theorem~\ref{power-of-2-special} to define in Section \ref{morestab} a
notion of the ``stable defect'' and ``stable length'' of a number $n$ -- these
notions measure what the defect and the addition chain length would be ``if $n$
were stable''.

The two theorems above are analogues for addition chains of results this author
previously showed for another notion called integer complexity \cite{cpxwo,
paper1} which has its own measure of defect.  In Section \ref{seccpx} we
discuss integer complexity, define its associated notion of defect $\delta(n)$,
and compare and contrast it with addition chain length.  Integer complexity
has the feature that it is definable by a dynamic programming recursion, and
this feature played an important role in the proof of well-ordering for defect
values in \cite{cpxwo}.  In contrast addition chain length is apparently
not definable by dynamic programming recursion, and the proofs here require new
ideas.

The proof of the main result for addition chains
works in much greater generality, and we will obtain Theorem
\ref{frontpagethmspecial} as a special case of Theorem \ref{frontpagethm} below.

\subsection{Methods}\label{methods} 

A key result which substitutes for dynamic programming and allows well ordering
to the proved in the addition chain case is 
the following result of Sch\"onhage \cite{schon}:

\begin{thm}[Sch\"onhage]
\label{schonthm}
For any $n \ge 1$,
\[ 
\dftl(n) \ge \log_2\nu_2(n) - \Cs, 
\]
where
\[ 
\Cs := \frac{2}{3} + \frac{2}{3}\log_2 3 - \frac{1}{\log 2} - \log_2 \log
\frac{4}{3} + \sum_{k=0}^\infty \log_2 (1+2^{-6\cdot 2^k + 1}) \le 2.13.
 \]
\end{thm}

The proof of Theorem \ref{frontpagethm} (and hence of Theorem
\ref{frontpagethmspecial}) requires only the assertion that $\dftl(n)$ can be
bounded below by some increasing unbounded function of $\nu_2 (n)$.  In fact,
similar but weaker inequalities were proven earlier by E.~G.~Thurber
\cite{thurberduke} and A.~Cottrell \cite{cottrell}.  However we can use
Sch\"{o}nhage's inequality to prove more detailed information on defect values;
see Theorem~\ref{rwo1} and Corollary~\ref{omegakbds}.

The idea of the proof is to consider initial segments of $\mathscr{D}^\acl$, say
$\mathscr{D}^\acl\cap[0,r]$.  By Theorem~\ref{schonthm}, numbers of bounded
defect have boundedly many $1$'s in their binary expansion.  But as we will show
in Proposition~\ref{wo2}, the set of defects arising from numbers with exactly
$k$ occurrences of $1$ in their binary expansion is well-ordered and has order
type at least $\omega^{k-1}$ but less than $\omega^k$.  From this fact we can
conclude (Theorem~\ref{rwo1}) that $\mathscr{D}^\acl\cap[0,r]$ is well-ordered
and has order type less than $\omega^\omega$, and thence that $\mathscr{D}^\acl$
itself is well-ordered with order type at most $\omega^\omega$.  To get the
lower bound on the order type, we note that $\mathscr{D}^\acl$ includes, for
every $k$, the set of defects arising from numbers with exactly $k$ occurrences
of $1$ in their binary expansion; by above, this means its order type must be at
least $\omega^k$ for every natural $k$, and hence at least $\omega^\omega$.

It is worth noting here that Sch\"onhage's inequality was proved as a partial
result working towards the following conjecture of Knuth and Stolarsky
\cite{TAOCP2, stolarsky, subreview}:
\begin{conj}[Knuth, Stolarsky]
\label{ksconj}
For all $n$, $s(n) \ge \log_2 \nu_2(n)$.
\end{conj}
It is possible that better understanding of the set of addition chain defects
could lead to a proof of this conjecture.

\subsection{Extensions and variations of the main theorem}
\label{variations}

The discussion above treated the addition chain length of $n$, but the
theorems can be proved more generally for other, similar notions of addition
chain complexity that put further restrictions on the allowed set $A$ of
addition chains.
A common variation on the notion of addition chains is the notion of the
\emph{star chain}; a star chain is an addition chain $(a_0,\ldots,a_r)$ with the
additional restriction that for any $k\ge 1$, there exists $i<k$ such that
$a_k=a_{k-1}+a_i$.  The length of the shortest star chain for $n$, called the
\emph{star chain length} of $n$, is denoted by $\acl^*(n)$. Naturally
$\acl^{*}(n) \ge \acl(n)$, and it is known that $\acl^{*}(n) \sim \log_2 n$.  We will see below that the results of this paper
apply to star chain length as well as addition length.  Indeed, we can
generalize much further.

Let $A$ be a fixed set of addition chains, such as the set of all addition
chains or the set of star chains.  We will be considering the length of the
shortest addition chain in $A$ for a number $n$; we denote this length by
$\acl^A(n)$.  However we will  not allow $A$ to be an arbitrary set of addition
chains, but require it to satisfy the following admissibility condition.

\begin{defn}
We define a set $A$ of addition chains to be \emph{admissible} if
\begin{enumerate}
\item For any $n$, there is an addition chain in $A$ for $n$ of length at most
$\lfloor \log_2 n \rfloor + \nu_2(n) - 1$.  That is to say, $\acl^A(n)$ is
defined and is at most $\lfloor \log_2 n \rfloor + \nu_2(n) - 1$.
\item For any $n$, $\acl^A(2n)\le \acl^A(n)+1$.
\end{enumerate}
\end{defn}

The first of these conditions says that for any $n$, there are chains in $A$ for
$n$ which are at least as short as those produced by the binary method.  So,
for instance, if $A$ includes all chains produced by the binary method, it
satisfies the first condition.  The meaning of the second condition is
straightforward.  It is is satisfied if, for instance, given any chain in $A$
for $n$, appending $2n$ again yields a chain in $A$, or if given any chain in
$A$ for $n$, doubling all the entries and prepending $1$ again yields a chain in
$A$.

Interesting examples of admissible sets of addition chains include: 
\begin{enumerate}
\item the set of all addition chains; 
\item the set of star chains;
\item the set of Hansen chains (also known as $\acl^0$-chains, see  Hansen
\cite{Hansen}, also \cite{TAOCP2, subreview});
\item
the set of chains which are star or quasi-star (see Subbarao \cite{subreview}).
\end{enumerate}
Of course, there are trivial examples as well.  For instance, one could let be
$A$ be just the set of addition chains produced by the binary method; then one
would always have $\acl^A(n)=\lfloor \log_2 n \rfloor + \nu_2(n) - 1$.  But the
particular set of addition chains chosen will mostly not matter so long as it
satisfies those two conditions.

One interesting set of addition chains that has been studied but which is not
admissible is the set of {\em Lucas chains}, also known as \emph{LUC chains};
they satisfy the second condition but not the first.  (For instance, the
shortest Lucas chain for $17$ has length $6$.)  See Kutz \cite{Kutz} for more
information on these.

Unless stated otherwise, we assume throughout that $A$ is an admissible set of
addition chains.  We can now make definitions analogous to those above with
$\acl^A$ replacing $\acl$:

\begin{defn}
For an admissible set $A$ of addition chains, we define the {\em $A$-defect}
\[ 
\dftA(n) := \acl^A(n) - \log_2 n. 
\]
If $A$ is the set of all addition chains, we just write $\dftl(n)$.  If $A$
is the set of star chains, we write $\dft^{*}(n)$.
\end{defn}

\begin{defn}
For an admissible set $A$ of addition chains, we define
\[ \mathscr{D}^A = \{ \dftA(n) : n\in \N \}. \]
If $A$ is the set of all addition chains, we just write $\mathscr{D}^\acl$.
If $A$ is the set of star chains, we write $\mathscr{D}^*$.
\end{defn}

With these, we can once again define:

\begin{defn}
A number $m$ is called \emph{$A$-stable} if $\acl^A(2^k m)=k+\acl^A(m)$ holds
for every $k \ge 0$.  Otherwise it is called \emph{$A$-unstable}.  If $A$ is the
set of all addition chains, we write \emph{$\acl$-stable}.  If $A$ is the set of
star chains, we write \emph{$*$-stable}.
\end{defn}

And with these, we once again get:

\begin{thm}
\label{power-of-2}
{\em ($A$-stability theorem)} 
Fix an admissible set $A$ of addition chains.  Then we have:
\begin{enumerate}
\item If $\alpha$ is a value of $\dftA$, and
\[
S(\alpha) := \{m:~~\dftA(m) = \alpha\}
\]
then there is a unique integer $n$ such that $S(\alpha)$ has either the form
$\{ n\cdot 2^k: 0 \le k \le K\}$ for some finite $K$ or else the form $\{ n\cdot
2^k: k\ge 0\}$. The integer $n$ will be  called the {\em leader} of
$S(\alpha)$.
\item The set $S(\alpha)$ is infinite if and only if $\alpha$ is the smallest
defect occurring among all defects $\dftl(2^k n)$ for $k\ge0$, where
$n$ is the leader of $S(\alpha)$.
\item For a fixed odd integer $n$, the sequence $\{ \dftA(n \cdot 2^k): \, k \ge
0\}$ is non-increasing.  This sequence takes on finitely many values, all
differing by integers, culminating in a smallest value $\alpha$ such that if
$\dftA(m)=\alpha$ and
$k\ge 0$, then
\[ \acl^A(m\cdot 2^k) = \acl^A(m) + k.\]
\end{enumerate}
\end{thm}

Another interesting variation on the set $\mathscr{D}^\acl$ or $\mathscr{D}^A$
is to restrict to defects of stable numbers.  We make the following definition:

\begin{defn}
\label{stabdftdef}
We define an \emph{$A$-stable defect} to be the defect of an $A$-stable number,
and define $\mathscr{D}^A_\st$ to be the set of all $A$-stable defects.
\end{defn}

This double use of the word ``stable'' could potentially be ambiguous if we had
a positive integer $n$ which were also a defect.  However, we will see
(Corollary~\ref{intdft}) that only integer which occurs as a defect is $0$, and
so this does not occur.

With these definitions, we obtain:

\begin{thm}
\label{frontpagethm}
{\em ($A$-defect well ordering theorem)} 
For any admissible set $A$ of addition chains,  
the sets $\mathscr{D}^A$ and $\mathscr{D}^A_\st$ are well-ordered subsets of
$\R$, of order type $\omega^\omega$.  In particular, the sets
$\mathscr{D}^\acl$, $\mathscr{D}^*$, $\mathscr{D}^\acl_\st$, and
$\mathscr{D}^*_\st$ are well-ordered, with order type $\omega^\omega$.
\end{thm}

We remark that Sch\"{o}nhage's lower bound theorem plays the same role in
establishing these well-ordering results as it does in the special case of all
addition chains, since $\dftA(n) \ge \dftl(n)$.

\subsection{Generalizations and open problems}

A natural generalization of addition chains is {\em
addition-subtraction chains}, where subtraction of two elements is permitted as
an elementary operation; the {\em addition-subtraction chain length} of $n$ is denoted $\acl^\pm(n)$.
Sch\"onhage \cite{schon} proved a lower
bound for addition-subtraction chains analogous to that in
Theorem~\ref{schonthm}.  However, our well-ordering result given  in Theorem
\ref{frontpagethmspecial} does not generalize to addition-subtraction chains.
Indeed, one can verify that for $k\ge 3$,
\[ 
\acl^\pm(2^k-1) = k+1;
\]
thus, if one were to define the {\em addition-subtraction chain defect} 
\[
\dft^\pm(n):=\acl^\pm(n)-\log_2 n, 
\]
then one would find that the image of this function contains the infinite
decreasing sequence $1-\log_2(1-2^{-k})$.
It follows that the set of all addition-subtraction chain defects is not well
ordered with respect to the usual ordering of the real line. 

Secondly, our proof of the well ordering in Theorem \ref{frontpagethmspecial}
does not currently enable us to determine all the cutoff values $c_k$ such that
the set of defect values $\mathscr{D}^{\acl} \cap [1, c_k)$ 
is of order type $\omega^k$.  In Section~\ref{smallk}, using the known
classification of numbers with $s(n)=1$ due to Gioia et al.~\cite{sub1962} and
of numbers with $s(n)=2$ due to Knuth \cite{TAOCP2}, we determine the cutoff
values for $k=1$ and $k=2$ to be $c_1=1$ and $c_2=2$ respectively.  (Recall that
$s(n)$ denotes $\lceil \dftl(n) \rceil$).  In Remark \ref{rmk47} we discuss
problems with determining values of $c_k$ for higher $k$.

Thirdly, in the integer complexity case there exists an effectively computable
algorithm for determining whether a given integer $n$ is stable (see
\cite{seq2}).
The methods of this paper do not give an effectively computble algorithm to test
if a given number is $\acl$-stable.  Finding such an algorithm remains an open
problem.

\section{Comparison of addition chain length and integer complexity}
\label{seccpx}

The main results in this paper are analogues for addition chains of results
recently established for integer complexity.  The  \emph{(integer) complexity}
of a natural number $n$ is the least number of $1$'s needed to write $n$ using
any combination of addition and multiplication, with the order of the operations
specified using parentheses grouped in any legal nesting.  For instance, $n=11$
has a complexity of $8$, since it can be written using $8$ ones as
$(1+1+1)(1+1+1)+1+1$, but not with any fewer.  This notion was
introduced in 1953 by Kurt Mahler and Jan Popken \cite{MP}, and more thoroughly
considered by Richard Guy \cite{Guy}.  We denote the complexity of $n$ by
$\cpx{n}$.

The parallel results for integer complexity  stem from a series of conjectures
formulated in 2000 by J. Arias de Reyna \cite{Arias}.  They include a conjecture
on stability for integer complexity, subsequently proved in 2012 by the author
with J.  Zelinsky \cite{paper1}. That paper introduced a notion of {\em (integer
complexity) defect}
\[
\delta(n) := ||n|| - 3 \log_3 n,
\]
and proved stability using that notion.  Some of Arias de Reyna's other
conjectures were reformulated by the author in terms of a well-ordering of the
values of the defect $\delta(n)$ for integer complexity, and a theorem
establishing the well-ordering of the range of the defect function was recently
proved by the author in \cite{cpxwo}.

In this section we expand on this analogy between integer complexity and
addition chain length.
These notions have obvious similarities; each is a measure of the
resources required to build up the number $n$ starting from $1$.  Both
allow the use of addition, but integer complexity supplements this by allowing
the use of multiplication, while addition chain length supplements this by
allowing the reuse of any number at no additional cost once it has been
constructed.  Furthermore, both measures are approximately logarithmic; integer
complexity satisfies the bounds

\[
3 \log_3 n= \frac{3}{\log 3} \log  n\le \cpx{n} \le \frac{3}{\log 2} \log n  ,\qquad n>1.
\]

A difference worth noting is that  while $\acl(n)$ is known to be
asymptotic to $\log_2 n$ as mentioned above, the function $\cpx{n}$ is not known
to be asymptotic to $3\log_3 n$; the value of the quantity $\limsup_{n\to\infty}
\frac{\cpx{n}}{\log n}$ remains unknown.  Guy \cite{Guy} has asked whether
$\cpx{2^k}=2k$ for $k\ge 1$; if true, it would make this quantity at least
$\frac{2}{\log 2}$. It is known that $\cpx{2^k}=2k$ does hold for $1 \le k \le
48$; see \cite{seq2}.

Another difference worth noting is that integer complexity, unlike addition
chain length, can be computed via dynamic programming.  Specifically, for any
$n>1$,

\begin{displaymath}
\cpx{n}=\min_{\substack{a,b<n\in \mathbb{N} \\ a+b=n\ \mathrm{or}\ ab=n}}
	\cpx{a}+\cpx{b}.
\end{displaymath}

By contrast, addition chain length is harder to compute.  Suppose we have a
shortest addition chain $(a_0,\ldots,a_{r-1},a_r)$ for $n$; one might hope that
$(a_0,\ldots,a_{r-1})$ is a shortest addition chain for $a_{r-1}$, but this
need not be the case.  An example is provided by the addition chain
$(1,2,3,4,7)$; this is a shortest addition chain for $7$, but $(1,2,3,4)$ is
not a shortest addition chain for $4$, as $(1,2,4)$ is shorter.   Moreover,
there is no way to assign to each natural number $n$ a shortest addition chain
$(a_0,\ldots,a_r)$ for $n$ such that $(a_0,\ldots,a_{r-1})$ is the addition
chain assigned to $a_{r-1}$ \cite{TAOCP2}. This can be an obstacle both to
computing addition chain length and proving statements about addition chains.

Nevertheless, this paper demonstrates there are important similarities between
integer complexity and addition chains.  The stabilization result
Theorem~\ref{power-of-2} is analogous to Theorem~5 from \cite{paper1}.  The
well-ordering result Theorem~\ref{frontpagethm} is analogous to part of
Theorem~1.3 from \cite{cpxwo}.  It is substantially weaker than a direct
analogue of Theorem~1.3, since it does not tell us where the supremum of the
initial $\omega^k$ defects occurs.  We prove bounds on this at the end of
Section~\ref{secwo} and in Section~\ref{smallk}.  We suspect that the supremum
of the initial $\omega^k$ defects is $k$, at least for addition chains; see
Conjecture~\ref{omegakk} and Question~\ref{genomegakk}.

\section{The $A$-defect and $A$-stabilization}
\label{secdft}

We will give proofs in this paper for an arbitrary admissible set $A$
of addition chains. 

\subsection{ $A$-defect}

The $A$-defect is the basic object of study in this paper.

\begin{prop}
\label{multdft}
Let $A$ be an admissible set of addition chains. We have
\begin{enumerate}
\item For all integers $a \ge 1$,
\[ 
\dftA(a) \ge 0.
\]
Here equality holds precisely when $a=2^k$ for some $k\ge 0$.

\item For $k\ge 0$,
\[
\dftA(2^k n) \le \dftA(n).
\]
The difference is an integer, and equality holds if and only if 
$$\acl^A(2^k n)= \acl^{A}(n)+k.
$$
\end{enumerate}
\end{prop}

\begin{proof}
The first statement in part (1) is just the lower bound $\acl^A(n)\ge \log_2 n$.
And for $n=2^k$, we know that $\acl^A(n)=k$, so $\dftA(n)=0$.  For the converse,
note that $\log_2 n$ is only an integer if $n$ is a power of $2$.

For part (2), note that by the requirements on $A$ we have
\begin{equation}
\label{pow2ineq}
\acl^A(2^k n)\le k+\acl^A(n).
\end{equation}
Subtracting $k+\log_2 n$ from both sides yields the stated
inequality.  Furthermore, since \eqref{pow2ineq} is an inequality of integers,
the difference is an integer; and we have equality in the result if and only if
we had equality in \eqref{pow2ineq}.
\end{proof}

As was noted in Section~\ref{sec10}, though one might hope that
$\acl(2n)=\acl(n)+1$ in general, infinitely many counterexamples are known
\cite{thurber}.  Still, based on this idea, we defined in
Section~\ref{sec10} the notions of an \emph{$\acl$-stable number} and in
Section~\ref{variations} the notion of an \emph{$A$-stable number}.

This can be alternately characterized as follows:

\begin{prop}
\label{stabdft}
The number $m$ is $A$-stable if and only if $\dftA(2^k m)=\dftA(m)$ for all
$k\ge 0$.
\end{prop}

\begin{proof}
This is immediate from Proposition~\ref{multdft}(2).
\end{proof}

This is already enough to prove the following:

\begin{thm}
\label{cj1}
We have
\begin{enumerate}
\item For any $m \ge 1$, there exists a finite $K\ge 0$ such that
$2^K m$ is $A$-stable.
\item If the defect $\dftA(m)$ satisfies $0 \le \dftA(m)<1$, then $m$ itself is
$A$-stable.
\end{enumerate}
\end{thm}

\begin{proof}
(1) From Proposition~\ref{multdft}, we have that for any $n$,
it holds that $\dftA(2n)\le \dftA(n)$, with equality if and only if $\acl^A(2n)=\acl^A(n)+1$.
More generally,
\[\dftA(n)-\dftA(2n)=\acl^A(n)+1-\acl^A(2n),\] and so the difference
$\dftA(n)-\dftA(2n)$ is always an integer.  This means that the sequence
$\dftA(m), \dftA(2m), \dftA(4m), \ldots$ is non-increasing, nonnegative, and can
only decrease in integral amounts; hence it must eventually stabilize. Applying
Proposition~\ref{stabdft} proves the theorem.

(2) If $\dftA(m)<1$, since all $\dftA(n) \ge 0$ there is no room to remove any
integral amount, so $m$ must be $A$-stable.
\end{proof}

Note that while this proof shows that for any $n$ there is some $K$ such that
$2^K n$ is $A$-stable (in particular, $\acl$-stable or $*$-stable), it does not
give any upper bound on $K$.

Because we use the actual logarithm, the value of the defect is enough to
determine a number up to a power of $2$:

\begin{prop}
\label{eqdefect}
Suppose that $m$ and $n$ are two positive integers, with $m\ge n$.  If $q:=
\dftA(n)- \dftA(m)$ is rational, then it is necessarily a nonnegative integer,
and furthermore $m=n \cdot 2^k$ for some $k \ge 0$.  In particular this holds if
$\dftA(n)=\dftA(m)$.
\end{prop}

\begin{proof}
If $q=\dftA(n)-\dftA(m)$ is rational, then $\log_2(m/n)$ is rational; since
$m/n$ is rational, the only way this can occur is if $\log_2(m/n)$ is an integer
$k$, in which case, since $m>n$, we have $m = n \cdot 2^k$ with $k \ge 0$.  It then
follows from the definition of defect that $q=\acl^A(n)+k-\acl^A(m)$.
\end{proof}

\begin{cor}
\label{intdft}
No nonzero integer occurs as $\dftA(n)$ for any $n$.
\end{cor}

\begin{proof}
If $\dftA(n)\in \Z$, then $n=2^k$ for some $k\ge 0$ by
Proposition~\ref{eqdefect}; but then $\dftA(n)=0$.
\end{proof}

We can now prove Theorems~\ref{power-of-2} and \ref{power-of-2-special}:

\begin{proof}[Proof of Theorem~\ref{power-of-2}]
For part (3), the non-increasing assertion follows from part (2) of
Proposition~\ref{multdft}.  Also, part (1) of Theorem~\ref{cj1} implies that
eventually the sequence stabilize; hence it can take only finitely many values.

For part (1), the assertion about the form of $S(\alpha)$ follows from Proposition \ref{eqdefect}.  The rest, and part (2),
follows from the fact that $\dftA(2^k n)$ is nonincreasing as a function of $k$.
\end{proof}

\begin{proof}[Proof of Theorem~\ref{power-of-2-special}]
This is just Theorem~\ref{power-of-2} in the case when $A$ is the set of all
addition chains.
\end{proof}

\subsection{$A$-stable defects and $A$-stable length}
\label{morestab}

Knowing the defect of a number also tells us whether or not that number is
stable:

\begin{prop}
If $\dftA(n)=\dftA(m)$ and $n$ is $A$-stable, then so is $m$.
\end{prop}

\begin{proof}
Suppose $\dftA(n)=\dftA(m)$ and $n$ is $A$-stable.  Then we can write $m=2^k
n$ for some $k\in \mathbb{Z}$.  Now, a number $a$ is $A$-stable if and only
if $\dftA(2^j a)=\dftA(n)$ for all $j\ge 0$; so if $k\ge 0$, then $m$ is
$A$-stable.  While if $k<0$, then consider $j\ge 0$; if $j\ge -k$, then
$\dftA(2^j m)=\dftA(2^{j+k} n)=\dftA(n)$, while if $j\le -k$, then
$\dftA(n)\le \dftA(2^j m)\le \dftA(m)$, so $\dftA(2^j m)=\dftA(m)$; hence $m$
is $A$-stable.
\end{proof}

Because of this proposition, Definition~\ref{stabdftdef} makes more sense; a
stable defect is not just the defect of a stable number, but one for which all
numbers with that defect are stable.

\begin{prop}
\label{modz1}
A defect $\alpha$ is $A$-stable if and only if it is the smallest
$\beta\in\mathscr{D}^A$ such that $\beta\equiv\alpha\pmod{1}$.
\end{prop}

\begin{proof}
This follows from part (2) of Proposition~\ref{multdft},
Proposition~\ref{eqdefect}, and part (1) of Theorem~\ref{power-of-2}.
\end{proof}

\begin{defn}
For a positive integer $n$, define the \emph{stable defect of $n$ with regard to
$A$}, denoted $\dftA_\st(n)$, to be $\dftA(2^k n)$ for any $k$ such that $2^k n$
is $A$-stable.  (This is well-defined as if $2^k n$ and $2^j n$ are $A$-stable,
then $k\ge j$ implies $\dftA(2^k n)=\dftA(2^j n)$, and so does $j\ge k$.)
\end{defn}

Here are two equivalent characterizations of stable defect:

\begin{prop}
\label{staltchar}
The number $\dftA_\st(n)$ can be characterized by:
\begin{enumerate}
\item $\dftA_\st(n)= \min_{k\ge 0} \dftA(2^k n)$
\item $\dftA_\st(n)$ is the smallest $\alpha\in\mathscr{D}^A$ such that
$\alpha\equiv \dft(n) \pmod{1}$.
\end{enumerate}
\end{prop}

\begin{proof}
Part (1) follows from part (2) of Theorem~\ref{multdft} and the fact that $m$
is $A$-stable if and only if $\dftA(2^k m)=\dftA(m)$ for all $k\ge 0$.  To
prove part (2), take $k$ such that $2^k n$ is $A$-stable.  Then $\dftA(2^k
n)\equiv \dftA(n) \pmod{1}$, and it is the smallest such by
Proposition~\ref{modz1}.
\end{proof}

So we can think about $\mathscr{D}^A_\st$ either as the subset of
$\mathscr{D}^A$ consisting of the $A$-stable defects, or we can think of it as
the image of $\dftA_\st$.  This double characterization will be useful in
Section~\ref{secwo}.

Just as we can talk about the stable defect of a number $n$, we can also talk
about its \emph{stable length} -- what the length of $n$ would be ``if $n$ were
stable''.

\begin{defn}
For a positive integer $n$, we define the \emph{stable length of $n$ with regard
to $A$}, denoted $\acl^A_\st(n)$, to be $\acl^A(2^k n)-k$ for any $k$ such that
$2^k n$ is $A$-stable.  This is well-defined; if $2^k n$ and $2^j n$ are both
stable, say with $k\le j$, then
\[\acl^A(2^k n)-k=k-j+\acl^A(2^j n)-k=\acl^A(2^j n)-j.\]
\end{defn}

\begin{prop}
\label{stabcpxprops}
We have:
\begin{enumerate}
\item $\acl^A_\st(n) = \min_{k\ge 0} (\acl^A(2^k n)-k)$
\item $\dftA_\st(n)=\acl^A_\st(n)-\log_2 n$
\end{enumerate}
\end{prop}

\begin{proof}
To prove part (1), observe that $\acl^A(2^k n)-k$ is nonincreasing in $k$, since
$\acl^A(2m)\le1+\acl^A(m)$.  So a minimum is achieved if and only if for all
$j$,
\[\acl^A(2^{k+j} n)-(k+j)=\acl^A(2^k n)-k,\]
i.e., for all $j$, we have
$\acl^A(2^{k+j} n)=\acl^A(2^k n)+j$, i.e., $2^k n$ is $A$-stable.

To prove part (2), take $k$ such that $2^k n$ is $A$-stable.  Then
\[\dftA_\st(n)=\dftA(2^k n)=\acl^A(2^k n)-\log_2(2^k n)=\acl^A(2^k n)-k-\log_2 n
=\acl^A_\st(n)-\log_2 n.\]
\end{proof}

\begin{prop}
\label{stabisstab}
We have:
\begin{enumerate}
\item $\dftA_\st(n) \le \dftA(n)$, with equality if and only if $n$ is
$A$-stable.
\item $\acl^A_\st(n) \le \acl^A(n)$, with equality if and only if $n$ is
$A$-stable.
\end{enumerate}
\end{prop}

\begin{proof}
The inequality in part (1) follows from Proposition~\ref{staltchar}.  Also, if
$n$ is $A$-stable, then for any $k\ge 0$, we have $\dftA(2^k n)=\dft(n)$, so
$\dftA_\st(n)=\dftA(n)$.  Conversely, if $\dftA_\st(n)=\dftA(n)$, then by
Proposition~\ref{staltchar}, for any $k\ge 0$, we have $\dftA(2^k n)\ge \dftA(n)$.  But
also $\dftA(2^k n)\le \dftA(n)$ by part (2) of Theorem~\ref{multdft}, and so
$\dftA(2^k n)=\dftA(n)$ and $n$ is $A$-stable.

Part (2) follows from part (1) along with part (2) of
Proposition~\ref{stabcpxprops}.
\end{proof}

\section{Bit-counting in numbers of small defect}
\label{bitcounting}

Sch\"onhage's Theorem, Theorem~\ref{schonthm}, implies that for any real $r\ge
0$, there is an upper bound on how many $1$'s can appear in the binary expansion
of a number with addition chain defect at most $r$.  Because of this, we
define:

\begin{defn}
We define a function $\bc:[0,\infty)\to \N$ by
\[ \bc(r) = \max_{\dftl(n)\le r} \nu_2(n). \]
More generally, for an admissible set of addition chains $A$, we can define
\[ \bc^A(r) = \max_{\dftA(n)\le r} \nu_2(n). \]
\end{defn}

Then in this language, Theorem~\ref{schonthm} says the following:

\begin{prop}
\label{schonreform}
For $r\ge 0$,
\[ \bc(r) \le \lfloor 2^{r+\Cs} \rfloor. \]
\end{prop}

\begin{proof}
Solving Theorem~\ref{schonthm} for $\nu_2(n)$ yields the inequality 
$\nu_2(n) \le
2^{\dftl(n) + \Cs}$; since $\nu_2(n)$ is an integer, it follows that 
$\nu_2(n) \le \lfloor 2^{\dftl(n) + \Cs} \rfloor$.  Hence, 
$\bc(r) \le \lfloor 2^{r+\Cs} \rfloor$.
\end{proof}

Note, by the way, the following properties of $\bc^A(r)$:
\begin{prop}
\label{bcprops}
Let $A$ and $B$ be admissible sets of addition chains. We have:
\begin{enumerate}
\item The function $\bc^A(r) $ is nondecreasing in real $r\ge 0$.
\item For $B\subseteq A$ and any $r$, we have $\bc^B(r)\le \bc^A(r)$.  In particular,
$\bc^A(r) \le \bc(r)$.
\end{enumerate}
\end{prop}

\begin{proof}
To prove part (1), observe that as $r$ increases, the set $\{n : \dftA(n)\le
r\}$ gets larger, and hence so does $\bc^A(r)$ as it is a maximum taken over
that set.  To prove part (2), note that for any $n$, we have $\dftA(n)\le \dft^B(n)$ and
so the set $\{n : \dft^B(n)\le r\}$ is contained in the set $\{n: \dftA(n)\le
r\}$; thus $\bc^A(r)$ is at least as large as it is a maximum over a superset.
\end{proof}

As was mentioned in Section~\ref{methods}, Sch\"onhage was not the first to
investigate the relation between $\nu(n)$ and $\dftl(n)$ -- or rather, between
$\nu(n)$ and $s(n)$, since $s(n)$ rather than $\dftl(n)$ has been the primary
object of study of previous authors.  Specifically, Sch\"onhage's theorem is a
partial result towards the Knuth-Stolarsky conjecture (Conjecture~\ref{ksconj})
that $s(n)\ge\log_2 \nu_2(n)$.

The Knuth-Stolarsky conjecture is known to be true for $0 \le s(n) \le 3$.  The
case $s(n)=0$ is trivial; the case $s(n)=1$ was proved by Gioia
et al.~\cite{sub1962}; the case $s(n)=2$ was proved by Knuth \cite{TAOCP2}; and
the case $s(n)=3$ was proved by Thurber \cite{thurber}.  In fact, Knuth proved a
more detailed theorem about the case $s(n)=2$; we will make use of this in
Section~\ref{parameters}.
We summarize these results formally here:

\begin{thm}[Gioia et al., Knuth, Thurber]
\label{gkt}
We have:
\begin{enumerate}
\item For a natural number $n$, $s(n)=0$ if and only if $\nu_2(n)=1$.
\item For a natural number $n$, $s(n)=1$ if and only if $\nu_2(n)=2$.
\item For a natural number $n$, if $s(n)=2$, then $\nu_2(n)=3$ or $\nu_2(n)=4$.
\item For a natural number $n$, if $s(n)=3$, then $\nu_2(n)\le 8$.
\end{enumerate}
\end{thm}

This theorem yields:

\begin{prop}
\label{gktreform}
For $k$ an integer with $0\le k\le 3$, $\bc(k)=2^k$.
\end{prop}

\begin{proof}
For $0\le k\le 3$ an integer, if $\dftl(n)\le k$, then $\nu_2(n)\le 2^k$ by
Theorem~\ref{gkt}.  That is to say, $\bc(k) \le 2^k$.  For the converse, observe
that $s(1)=0$ and $\nu_2(1)=1$, so $\bc(0)\ge 1$; $s(3)=1$ and $\nu_2(3)=2$, so
$\bc(1)\ge 2$; $s(15)=2$ and $\nu_2(15)=4$, so $\bc(2)\ge 4$; and $s(255)=3$ and
$\nu_2(255)=8$, so $\bc(3)\ge 8$.
\end{proof}

So while Sch\"onhage's theorem yields the best known result for large $r$, these
results settle the matter for small $r$.

\begin{rem}
\label{rmk47}
In Section~\ref{secwo}, we will give an upper bound on the order type of
$\mathscr{D}^\acl \cap [0,r]$ in terms of $\bc(r)$.  So while in this paper
we state concrete bounds proved using Theorem~\ref{schonthm}, any improvement in
the upper bounds on $\bc(r)$ -- for instance, a proof of the Knuth-Stolarsky
conjecture -- would improve these bounds.  (Note that if one wants merely to
prove Theorem~\ref{frontpagethmspecial}, it suffices to know that $\bc(r)$ is
well-defined; one does not even need to know any bounds on it at all.)  However,
this does not mean that one is limited to bounds based on $\bc(r)$; in
Section~\ref{parameters}, we will demonstrate an example of a bound that goes
beyond what one can learn from study of $\bc(r)$ alone.
\end{rem}

\section{Cutting and pasting well-ordered sets}
\label{external}

We pause to recall some external facts dealing with the cutting and
pasting of well-ordered sets.  We begin with the following theorem of
P.~W.~Carruth \cite{carruth}:

\begin{thm}
\label{natsum}
Let $S$ be a well-ordered set and suppose $S=S_1 \cup S_2$.  Then the order type
of $S$ is at most the \emph{natural sum} of the order types of $S_1$ and $S_2$.
\end{thm}

The natural sum is defined as follows \cite{carruth}:

\begin{defn}
The \emph{natural sum} (also known as the \emph{Hessenberg sum}) \cite{wpo} of
two ordinals $\alpha$ and $\beta$, here denoted $\alpha \oplus \beta$, is
defined by simply adding up their Cantor normal forms as if they were
``polynomials in $\omega$''.  That is to say, if there are ordinals $\gamma_0 <
\ldots < \gamma_n$ and whole numbers $a_0, \ldots, a_n$ and $b_0, \ldots, b_n$
such that $\alpha = \omega^{\gamma_n}a_n + \ldots + \omega^{\gamma_0}a_0$ and
$\beta = \omega^{\gamma_n}b_n + \ldots + \omega^{\gamma_0}b_0$, then
\[ \alpha\oplus\beta = \omega^{\gamma_n}(a_n+b_n) + \ldots + 
	\omega^{\gamma_0}(a_0+b_0). \]
\end{defn}

Theorem~\ref{natsum} is sometimes used as the definition of the natural sum
\cite{carruth}.  There is also a recursive definition \cite{ONAG}.  There is
also a similar natural product \cite{carruth, wpo}, but we will not be using it
here.  See \cite{wpo} for generalizations of this theorem.

From this we can then conclude:

\begin{prop}
\label{cutandpaste}
For any ordinal $\alpha$:
\begin{enumerate}
\item If $S$ is a well-ordered set and $S=S_1\cup\ldots\cup S_n$, and $S_1$
through $S_n$ all have order type less than $\omega^\alpha$, then so does $S$.
\item If $S$ is a well-ordered set of order type $\omega^\alpha$ and
$S=S_1\cup\ldots\cup S_n$, then at least one of $S_1$ through $S_n$ also has
order type $\omega^\alpha$.
\end{enumerate}
\end{prop}

\begin{proof}
For (1), observe that the order type of $S$ is at most the natural sum of those
of $S_1,\ldots,S_n$, and the natural sum of ordinals less than $\omega^\alpha$
is again less than $\omega^\alpha$.

For (2), by (1), if $S_1,\ldots, S_k$ all had order type less than
$\omega^\alpha$, so would $S$; so at least one has order type at least
$\omega^\alpha$, and it necessarily also has order type at most $\omega^\alpha$,
being a subset of $S$.
\end{proof}

We can say more when the sets are interleaved with each other:
\begin{prop}
\label{interleave}
Suppose $\alpha$ is an ordinal and $S$ is a well-ordered set which can be
written as a finite union $S_1 \cup \ldots \cup S_k$ such that:
\begin{enumerate}
\item The $S_i$ all have order types at most $\omega^\alpha$.
\item If a set $S_i$ has order type $\omega^\alpha$, it is cofinal in $S$.
\end{enumerate}
Then the order type of $S$ is at most $\omega^\alpha$.  In particular, if at
least one of the $S_i$ has order type $\omega^\alpha$, then $S$ has order type
$\omega^\alpha$.
\end{prop}

\begin{proof}
Consider a proper initial segment of $S$; call it $T$.  Let $x$ be the smallest
element of $S\setminus T$.  Let $\mathcal{A}$ be the set of $S_i$ which have order type
$\omega^\alpha$.  Since each element of $\mathcal{A}$ is cofinal in $S$, each
contains some element that is at least $x$, and thus not in $T$.  That is, for
$S_i\in\mathcal{A}$, the set $T\cap S_i$ is always a proper initial segment of $S_i$.
Thus $T$ is a finite union of proper initial segments of the elements of
$\mathcal{A}$ and possibly improper initial segments of the $S_i$ not in
$\mathcal{A}$.  But any set with either of these order types has order type strictly
less than $\omega^\alpha$, and so by Proposition~\ref{cutandpaste}, so does
$T$.  Since each proper initial segment of $S$ has order type less than
$\omega^\alpha$, it follows that $S$ has order type at most $\omega^\alpha$.  If
furthermore some $S_i$ has order type $\omega^\alpha$, then $S$ also has order
type at least $\omega^\alpha$ and thus exactly $\omega^\alpha$.
\end{proof}

We'll be applying  these propositions to take apart and put together sets of defects in the
subsequent sections.

Also worth noting is the following fact.
\begin{prop}
\label{limitvsclos}
Let $X$ be a totally ordered set with the least upper bound property, and $S$ a
well-ordered subset of $X$ of order type $\alpha$.  Then $\overline{S}$ is a
well-ordered subset of $S$ of order type either $\alpha$ or $\alpha+1$, and if
$\beta<\alpha$ is a limit ordinal, the $\beta$'th element of $\overline{S}$ is
the supremum (limit) of the initial $\beta$ elements of $S$.
\end{prop}

\begin{proof} This result is proved in \cite{cpxwo} .
\end{proof}
\section{Well-ordering of defects}
\label{secwo}

Now we are prepared to prove that the set of defects is well-ordered.

\subsection{Well-ordering of defect sets for $n$ with $\nu_2(n) \le k$}

First we
observe:

\begin{prop}
\label{dftlbd}
For any $n$, $\dftA_\st(n)\le \dftA(n) \le \nu_2(n) - 1$.
\end{prop}

\begin{proof}
We know $\dftA_\st(n)\le \dftA(n)$ by Proposition~\ref{stabisstab}, and the
rest is immediate as
\[
 \dftA(n) = \acl^A(n) - \log_2 n
 \le \lfloor \log_2 n \rfloor - \log_2 n +\nu_2(n) - 1 
 \le \nu_2(n) - 1.
 \]
\end{proof}

Next we show that, applied to numbers with a fixed number of $1$'s in the binary
expansion, the binary method produces a well-ordered set of defects.

\begin{prop}
\label{wo1}
Let $k\ge 1$ be a natural number, and define the set $S_k$ to be
\[ \{ k-1+\lfloor \log_2 n \rfloor - \log_2 n: \nu_2(n)=k \}. \]
Then $S_k$ is a well-ordered set, with order type $\omega^{k-1}$.
\end{prop}

\begin{proof}
If $\nu(n)=k$, write $n=2^{a_0}+\ldots+2^{a_{k-1}}$.  Then $\lfloor \log_2
n\rfloor = a_0$ and
\[ k - 1 + \lfloor \log_2 n \rfloor - \log_2 n =
	k - 1 - \log_2(1+2^{a_1-a_0} + \ldots + 2^{a_{k-1}-a_0}). \]
We observe then that $S_k$ can also be written as
\[ \{ k-1-\log_2(1+2^{-b_1} + \ldots + 2^{-b_{k-1}}) :
	0<b_1<b_2<\ldots<b_{k-1} \in \mathbb{Z} \}. \]
This set contains $S_k$ as $a_0>a_i$ for $i>0$ and the sequence of $a_i$ is
decreasing, and the converse holds as, given $b_1,\ldots,b_{k-1}$, we can pick
$a_0=\sum_{i=1}^{k-1} b_i$ and $a_i=a_0-b_i$ for $i>0$.
Now we can write down an order-preserving bijection $\phi:\omega^{k-1}\to S_k$.
Define $\phi(c_1,\ldots,c_{k-1})=k-1-\log_2(1+2^{-b_1}+\ldots+2^{-b_{k-1}})$,
where
\[ b_i = i + \sum_{j=0}^{i} c_j. \]
This is a bijection as, since an element of $S_k$ is identified by its sequence
of $b_1,\ldots,b_{k-1}$, it has inverse given by
\[ c_i = b_i - b_{i-1} - 1 \]
(where we take $b_0=0$).
To see this is order-preserving, take
$(c_1,\ldots,c_{k-1})<(c'_1,\ldots,c'_{k-1})$; say $c_1=c'_1, \ldots,
c_i=c'_i$ and $c_{i+1}<c'_{i+1}$.  Then $b_j=b'_j$ for $1\le j\le i$ and
$b'_{i+1}>b_{i+1}$. So
\[ 2^{-b_1} + \ldots + 2^{-b_{k-1}} > 2^{-b'_1} + \ldots + 2^{-b'_{k-1}} \]
as they have the same binary expansion up to $2^{-b_i}$ place, but the former's
next $1$ occurs at $2^{-b_{i+1}}$, and the latter's next $1$ occurs at
$2^{-b'_{i+1}}$, and $b'_{i+1}>b_{i+1}$.
Since $k-1-\log_2(1+2^{-b_1} + \ldots + 2^{-b_{k-1}})$ is an order-reversing
function of $2^{-b_1} + \ldots + 2^{-b_{k-1}}$, this implies
$\phi(c_1,\ldots,c_{k-1})<\phi(c'_1,\ldots,c'_{k-1})$, proving the claim.
\end{proof}

Next we see that this is true even when chains may be shorter than those
produced by the binary method:

\begin{prop}
\label{wo2}
For $k\ge 1$, the set $\{\dftA(n) : \nu_2(n) = k\}$ is a well-ordered subset of
the real numbers, with order type at least $\omega^{k-1}$ and at most
$\omega^{k-1}k<\omega^k$.  The same is true of the set $\{\dftA_\st(n) : \nu_2(n)
= k\}$.
\end{prop}

\begin{proof}
We prove it here for the set $\{\dftA(n) : \nu_2(n) = k\}$; the proof for the
set $\{\dftA_\st(n) : \nu_2(n) = k\}$ is analogous.

Say $\nu_2(n)=k$, and write $n=2^{a_0}+\ldots+2^{a_{k-1}}$.  Then $\acl^A(n)\le
k-1 + a_0$, i.e., $\acl^A(n)=k-1+a_0-m$ for some integer $m\ge 0$.  So also
\[\dftA(n) = k-1 + a_0 - m - \log_2 n \le k-1 -m.\]
But also $\dftA(n)\ge 0$, so $m\le k-1$.  As $m$ is an integer, this means $m\in
\{0,\ldots,k-1\}$, a finite set.

So if we fix $k$ and let $T$ be the set $\{\dftA(n) : \nu_2(n) = k\}$ and $U$ be
the set $\{k-1 - \log_2 n : \nu_2(n)=k\}$, then we see that $T$ is covered by
finitely many translates of $S_k$ from Proposition~\ref{wo1}; more specifically,
we can partition $S_k$ into $U_0,\ldots,U_{k-1}$ such that
\[ T=U_0\cup U_1-1\cup \ldots\cup U_{k-1}-(k-1). \]
But by Proposition~\ref{wo1}, $S_k$ has order type $\omega^{k-1}$.  So the $U_i$
all have order type at most $\omega^{k-1}$, and by Proposition~\ref{cutandpaste}
at least one has order type $\omega^{k-1}$.  Hence $T$ is well-ordered of order
type at most $\omega^{k-1}k<\omega^k$ by Proposition~\ref{natsum}, and by above
it also has order type at least $\omega^{k-1}$.
\end{proof}

\subsection{Well-ordering of initial segment of $A$-defect set}

Finally we apply the existence of an upper bound on $\nu_2$ in terms of $\dftl$
to prove the theorem:

\begin{thm}
\label{rwo1}
{\em (Well-ordering of intial segments of $A$-defect set)}
Let $A$ be an admissible set of addition chains, 
and let $r\ge 0$ be a real number.  Then $\mathscr{D}^A \cap [0,r]$ is a
well-ordered subset of the real numbers with order type at least
$\omega^{\lfloor r \rfloor}$ and at most
\[ \omega^{\bc^A(r) -1}\bc^A(r) + \ldots + \omega^2 3 + \omega 2 + 1,\]
which is less than $\omega^{\bc^A(r)-1}(\bc^A(r)+1)$ and hence less than
$\omega^{\bc^A(r)}$.
The same is true of $\mathscr{D}^A_\st \cap [0,r]$.
\end{thm}

\begin{proof}
Say $n$ is a number with $\dftA(n)\le r$; then 
$\nu_2(n)\le \bc^A(r)$.  So
$\mathscr{D}^A\cap [0,r]$ can be covered by the sets 
$\{\dftA(n): \nu_2(n)=k\}$
for $k=1,2,\ldots,\bc^A(r)$.  By Proposition~\ref{wo2}, each of these sets is
well-ordered, with order type at most $\omega^{k-1}k$.  Hence by
Proposition~\ref{natsum}, $\mathscr{D}^A\cap[0,r]$ is well-ordered with order
type at most \[ \omega^{\bc^A(r) -1}\bc^A(r) + \ldots + \omega^2 3 + \omega 2 +
1,\] which is less than $\omega^{\bc^A(r)-1}(\bc^A(r)+1)$ and hence less than
$\omega^{\bc^A(r)}$.  Since $\mathscr{D}^A_\st \cap [0,r]$ is a subset of
$\mathscr{D}^A \cap [0,r]$, this upper bound applies to it as well.

For the lower bound, observe that the set $\{\dftA_\st(n): \nu_2(n)=\lfloor
r\rfloor+1 \}$ is, by Proposition~\ref{dftlbd}, entirely contained within
$\mathscr{D}^A_\st\cap[0,r]$, and by Proposition~\ref{wo2} it has order type at
least $\omega^{\lfloor r\rfloor}$, and thus so does $\mathscr{D}^A_\st \cap
[0,r]$, and so also does $\mathscr{D}^A \cap [0,r]$.
\end{proof}

If we plug in Theorem~\ref{schonreform}, we get an explicit version of this.  We
can also plug in the other bounds in Section~\ref{bitcounting} to yield explicit
versions of this that will be worse for large $r$ but sometimes better for small
$r$; see Section~\ref{smallk} for more on this.

We can now prove Theorem~\ref{frontpagethm}.

\begin{proof}[Proof of Theorem~\ref{frontpagethm}]
We prove the theorem for $\mathscr{D}^A$; the proof for $\mathscr{D}^A_\st$ is analogous.
Take an initial segment of $\mathscr{D}^A$, say $\mathscr{D}^A \cap [0,r)$.
Then this is contained in $\mathscr{D}^A \cap [0,r]$ and so well-ordered with
order type less than $\omega^{\bc^A(r)}$ by Theorem~\ref{rwo1}.  Hence
$\mathscr{D}^A$ is well-ordered with order type at most $\omega^\omega$, as all
its initial segments are well-ordered with order type less than $\omega^\omega$.
Furthermore, for any whole number $k$, $\mathscr{D}^A \cap [0,k]$ is
well-ordered with order type at least $\omega^k$ by Theorem~\ref{rwo1}, so
$\mathscr{D}^A$ must have order type at least $\omega^\omega$ as well.
\end{proof}

\begin{proof}[Proof of Theorem~\ref{frontpagethmspecial}]  This follows
immediately from Theorem~\ref{frontpagethm} by taking $A$ to be the set of all
addition chains.
\end{proof}

\subsection{Cutoff values $\posf^A(k)$  for $\omega^k$-limit points} 

We can turn the well-ordering question around and consider, what is the supremum
(limit) of the initial $\omega^k$ defects?  This is of course essentially the
same question, but it is also a helpful way of thinking about the question, so
we note the results here.

\begin{defn}
We define $\posf^A(k)$ to be the limit of the initial $\omega^k$ defects in
$\mathscr{D}^A$, and $\posf^A_\st(k)$ to be the limit of the initial $\omega^k$
defects in $\mathscr{D}^A_\st$.  Note that by Proposition~\ref{limitvsclos}, if
$k\ge 1$, this is the same as the $\omega^k$'th element of
$\overline{\mathscr{D}^A}$ (or $\overline{\mathscr{D}^A_\st}$), while if $k=0$,
this is the same as the $0$'th element of $\overline{\mathscr{D}^A}$ (or
$\overline{\mathscr{D}^A_\st})$.  If $A$ is the set of all addition chains we
will write $\posf^\acl$; if $A$ is the set of star chains we will write
$\posf^*$.
\end{defn}

\begin{prop}
For any $k$, we have $\posf^A(k) \le \posf^A_\st(k)$.
\end{prop}

\begin{proof}
The set $\mathscr{D}^A_\st$ is a subset of $\mathscr{D}^A$; hence for
$\alpha<\omega^\omega$, the $\alpha$'th element of $\mathscr{D}^A_\st$ is at
least the $\alpha$'th element of $\mathscr{D}^A$.  Taking limits,
$\posf^A_\st(k) \ge \posf^A(k)$.
\end{proof}

We now have the following corollary of Theorem~\ref{rwo1}:

\begin{cor}
\label{omegakbds}
We have
\[ \log_2(k+1)-2.13 < \log_2 (k+1) - \Cs < \posf^A(k) \le \posf^A_\st(k) \le k.
\]
\end{cor}

\begin{proof}
For the upper bound, observe that by Theorem~\ref{rwo1}, the order type of
$\mathscr{D}^A_\st \cap [0,k]$ is at least $\omega^k$, so
$\mathscr{D}^A_\st(\omega^k) \le k$.

For the lower bound, consider $\mathscr{D}^A\cap[0,r]$ with
$r<\log_2(k+1)-\Cs$.  Then $2^{r+\Cs}<k+1$, so $\lfloor 2^{r+\Cs} \rfloor \le
k$.  Since $\bc^A(r)\le \lfloor 2^{r+\Cs} \rfloor\le k$ by
Theorem~\ref{schonreform} and Proposition~\ref{bcprops}, by Theorem~\ref{rwo1},
$\mathscr{D}^A \cap [0,r]$ has order type less than $\omega^{k-1}(k+1)$.
Hence, if we consider $\mathscr{D}^A \cap [0,\log_2(k+1)-\Cs)$, all its proper
initial segments have order type less than $\omega^{k-1}(k+1)$, and so it has
order type at most $\omega^{k-1}(k+1)<\omega^k$.  Thus we must have
$\posf^A(k)> \log_2(k+1) -\Cs$.
\end{proof}

We will examine this question further in the next section, where we will improve
this for small $k$.

\section{Bounds on order type for small $A$-defect values }
\label{smallk}

In the previous section, we proved bounds on the order types of
$\mathscr{D}^A\cap[0,r]$ and $\mathscr{D}^A_\st \cap [0,r]$.  However, as was
noted in Section~\ref{bitcounting}, we can say more when $r$ is small.  First,
we note the implications of the theorems in Section~\ref{bitcounting} regarding
the functions $\posf^A$ and $\posf^A_\st$ defined in the previous section.  Then
we will perform a more detailed examination of the case $r\le 2$ using a theorem
of Knuth.  Then we compile these results to present bounds on $\posf^A(k)$ and
$\posf^A_\st(k)$ when $k$ is small.  We also make some notes on stability of
numbers of small defect.

\subsection{Bound for $A$-defect $r<1$}

The case of $r\le 1$ can be handled with part (2) of Theorem~\ref{gkt}, that was
proved by Gioia et al.

\begin{thm}
\label{keq1thm}
The order type of $\mathscr{D}^A\cap [0,1]$ is $\omega$, while for any $r<1$,
the set $\mathscr{D}^A\cap [0,1]$ is finite.  Furthermore, all defects in
$\mathscr{D}^A\cap [0,1]$ are $A$-stable, and so the order type of
$\mathscr{D}^A_{\st}\cap [0,1]$ is $\omega$.
\end{thm}

\begin{proof}
Suppose that $\lceil \dftA(n) \rceil=1$.  Then $\dftl(n)\le \dftA(n) \le1$, so
$\lceil \dftl(n) \rceil=1$ unless $n$ is a power of $2$, and $n$ cannot be a
power of $2$, as then we would have $\dftA(n)=0$.  So we can apply
Theorem~\ref{gkt} to conclude that $n$ can be written as $2^a+2^b$ for some
$b>a$.  Conversely, if $n=2^a+2^b$ with $b>a$, then $\acl^A(n)\le b+1$ by the
assumption that $A$ is admissible, and we cannot have $\acl^A(n)\le b$ as
otherwise we would have $\dftA(n)<0$; so $\acl^A(n)=b+1$.

Thus the set $\mathscr{D}^A \cap [0,1]$ is precisely $\{0\}\cup S_2$, where
$S_2$ is as in Proposition~\ref{wo1}.  Thus by that same proposition it has
order type $\omega$.  Also it is easily seen to have a supremum of $1$, so for
$r<1$,  the set $\mathscr{D}^A \cap [0,r]$ is a proper initial segment of
$\mathscr{D}^A \cap [0,1]$ and so has strictly smaller order type.

Furthermore, if $n=2^a+2^b$ with $b>a$, then $2^k n= 2^{a+k} + 2^{b+k}$, and so
$\acl^A(2^k n)=b+k+1=k+\acl^A(n)$, and so $n$ is $A$-stable.  While if $n=2^b$,
then $\acl^A(2^k n)=b+k=k+\acl^A(n)$, and so again $n$ is $A$-stable.  This
proves the stability part of the theorem.
\end{proof}

\subsection{Bounds for $A$-defect $r< 2$.}
\label{parameters}

For the case $k=2$, we will need to go beyond what is in Theorem~\ref{gkt}.  We
state here the full theorem regarding numbers with $2$ small steps, as proved by
Knuth \cite{TAOCP2}:

\begin{thm}[Knuth]
\label{keq2src}
For a positive integer $n$, $s(n)=2$ if and only if $n$ can be
written in one of the following forms:
\begin{enumerate}
\item $2^a+2^b+2^c$ for $0\le a<b<c$
\item $2^a + 2^{a+1} + 2^{a+2} + 2^{a+7}$ for $a\ge 0$
\item $2^a + 2^{a+1} + 2^b + 2^{b+3}$ for $b>a+1$, $a\ge 0$
\item $2^a + 2^b + 2^c + 2^{b+c-a}$ for $0\le a<b<c$
\item $2^a + 2^b + 2^c + 2^{b+c-a+1}$ for $0\le a<b<c$
\end{enumerate}
\end{thm}

With this, we can handle the case of $\mathscr{D}^A\cap[0,2]$ with an argument
which is similar to that of Theorem~\ref{keq1thm} but slightly more involved:

\begin{thm}
\label{keq2thm}
The order type of $\mathscr{D}^A\cap [0,2]$ is $\omega^2$, while for $r<2$,
the set $\mathscr{D}^A \cap[0,r]$ has order type strictly less than $\omega^2$.
Furthermore, all defects in $\mathscr{D}^A\cap [0,2]$ are $A$-stable, and so the
order type of $\mathscr{D}^A_{\st}\cap [0,2]$ is $\omega^2$.
\end{thm}

\begin{proof}
Suppose that $\lceil \dftA(n) \rceil=2$.  Then $\dftl(n)\le \dftA(n) \le2$, so
by Theorem~\ref{gkt}, $\lceil \dftl(n) \rceil=2$ unless $\nu_2(n)\le 2$, and
this cannot occur, as then we would have $\dftA(n)\le 1$.  So we can apply
Theorem~\ref{keq2src} to conclude that $n$ can be written in one of the forms
listed there.

Conversely, suppose we have a number $n$ of one of the forms listed in
Theorem~\ref{keq2src}.  Since $\nu_2(n)>2$, we have $\acl^A(n)>\floor{\log_2 n}+1$.  And if $\acl^A(n)\ge
\floor{\log_2 n}+3$, then
\[ \dftA(n) = \floor{\log_2 n}+3-\log_2 n > 2. \]
Thus, $\mathscr{D}^A\cap(1,2]$ is a subset of
\[ T:=\{ 2+\floor{\log_2 n}-\log_2 n: n\ \textrm{satisfies the conclusion of
Theorem~\ref{keq2src}} \}. \]
Let $T_i$ denote the set
\[ \{ 2+\floor{\log_2 n}-\log_2 n: n\ \textrm{falls under case $i$ of
Theorem~\ref{keq2src}} \}, \]
so that $T$ is the union of $T_1$ through $T_5$.  We will examine each of these
sets in turn.

The set $T_1$ is the same as the set $S_3$ from Proposition~\ref{wo1}, and so
has order type $\omega^2$.  In fact, if $n=2^a+2^b+2^c$, with $c>b>a$, then
$\acl^A(n)\le c+2$ by the assumption that $A$ is admissible, and so
$\acl^A(n)=c+2$ and $\dftA(n)<2$, meaning that all of $T_1$, rather than just a
subset, is contained in $\mathscr{D}^A\cap[0,2]$.  As was noted earlier, we can
rewrite $S_3$ as the set
\[ \{ 2-\log_2(1+2^{-a} + 2^{-b}) : 0<a<b\in\mathbb{Z}\}. \]
As $a$ and $b$ go to infinity, this expression goes to $2$, and so we see that
$\sup T_1 = 2$, and thus $T_1$ must be cofinal in $T\subseteq[0,2)$.

The set $T_2$ is easily seen to be equal to the set $\{9-\log_2 135\}$, which
has order type $1=\omega^0$.  This number is also strictly less than $2$ and so
$T_2$ is not cofinal in $T$.

The set $T_3$ is equal to the set $\{5-\log_2(9+3\cdot2^{-a}): a\ge2\}$, which
is a monotonic image of $\mathbb{N}$ and so has order type $\omega$.  It is also
bounded above by $5-\log_2 9<2$ and so not cofinal in $T$.

Finally, we consider the sets $T_4$ and $T_5$; we claim that both are order
isomorphic to $S_3$ and hence to $\omega^2$, and both are cofinal in $T$.  We
will only explicitly treat the case of $T_4$, as $T_5$ is similar.  First
observe that $T_4$ is equal to the set
\[ \{ 2-\log_2(1+2^{-a} + 2^{-b} + 2^{-a-b}) : 0<a<b\in\mathbb{Z}\}. \]
As $a$ and $b$ go to infinity, this expression approaches $2$, so $T_4$ is
cofinal in $T$.  To see that it has order type $\omega^2$, consider the map
\[ 2-\log_2(1+2^{-a}+2^{-b}) \mapsto 2-\log_2(1+2^{-a}+2^{-b}+2^{-a-b})\]
(where here $b>a>0$).  Let $f(a,b)$ denote $2-\log_2(1+2^{-a}+2^{-b})$ and
$g(a,b)$ denote $2-\log_2(1+2^{-a}+2^{-b}+2^{-a-b})$.  Then it is
straightforward to check that $f(a_1,b_1)>f(a_2,b_2)$ if and only if
$(a_1,b_1)>(a_2,b_2)$ lexicographically, which also is true if and only if
$g(a_1,b_1)>g(a_2,b_2)$.  Hence the map above, sending $f(a,b)$ to $g(a,b)$ is
an order isomorphism, proving the claim.  As mentioned above, the case of $T_5$
is similar.

Thus, by Proposition~\ref{interleave}, $T$ has order type $\omega^2$.  And so
$\mathscr{D}^A\cap(1,2]$ has order type at most $\omega^2$, and so
$\mathscr{D}^A\cap[0,2]$ has order type at most $\omega+\omega^2=\omega^2$.  We
also already know it has order type at least $\omega^2$, so it has order type
exactly $\omega^2$.

Also, the supremum of $\mathscr{D}^A\cap[0,2]$ is $2$, so for any $r<2$, the set
$\mathscr{D}^A\cap[0,r]$ is a proper initial segment and so has order type
strictly less than $\omega^2$.

Finally, note that if $\lceil \dftA(n) \rceil=2$, then $n$ must be $A$-stable,
since otherwise, there would be some $k$ with $\dftA(2^k n)<1$; but
$\nu_2(n)\ge 3$ and $\nu_2(2^k n)\le 2$, so this is impossible.  By
Theorem~\ref{keq1thm}, all defects in $\mathscr{D}^A\cap[0,1]$ are stable, and
by the above, all defects in $\mathscr{D}^A\cap(1,2]$ are stable, so the
stability part of the theorem follows.
\end{proof}

\subsection{Summing up: Lower bounds}

So we can now sum up the lower bounds on $\posf^A(k)$ and $\posf^A_\st(k)$ as
follows:

\begin{thm}
\label{posfsummary}
For $k$ a whole number, we have:
\begin{enumerate}
\item For $0\le k\le 2$, we have $\posf^A(k)=\posf^A_\st(k)=k$.
\item For $3\le k\le 7$, we have $2<\posf^A(k)\le\posf^A_\st(k)\le k$.
\item For $8\le k\le 33$, we have $3<\posf^A(k)\le\posf^A_\st(k)\le k$.
\item For $k\ge 34$, we have $\log_2(k+1)-\Cs<\posf^A(k)\le\posf^A_\st(k)\le k$.
\end{enumerate}
\end{thm}

\begin{proof}
The upper bounds are just Corollary~\ref{omegakbds}, so we focus on the lower
bounds.

For $k=0$, this follows as $0\in\mathscr{D}^\acl$.  For $k=1$, this is immediate
from Theorem~\ref{keq1thm}.  For $k=2$, this is immediate from
Theorem~\ref{keq2thm}.  Part (2) then follows as $\posf^A$ is strictly
increasing.

For part (3), observe that by Theorem~\ref{rwo1} and
Proposition~\ref{gktreform}, the order type of $\mathscr{D}^A\cap[0,3]$ is less
than $\omega^8$, and so $\posf^A(8)>3$; the rest then follows as $\posf^A$ is
increasing.  Finally, part (4) is just Corollary~\ref{omegakbds}.
\end{proof}

\section{Concluding Remarks}

In future papers we hope to  prove better bounds on $\posf^\acl(k)$,
$\posf^*(k)$, and their stable versions.  Meanwhile we conjecture:

\begin{conj}
\label{omegakk}
(1) For $k\ge 0$, $\posf^\acl(k)=\posf^\acl_\st(k)=k$.

(2) For $k \ge 0$,  $\posf^*(k)=\posf^*_\st(k)=k$.
\end{conj}

We can say for a fact that there are certain sets of addition chains $A$ for
which we know an analogue of  Conjecture~\ref{omegakk} holds; we could take $A$
to be the set of addition chains generated by the binary method.  Then we would
have $\mathscr{D}^A=\mathscr{D}^A_\st=\bigcup_{k\ge 1} S_k$, where $S_k$ is as
in Proposition~\ref{wo1}.  It is then easy to check that, for $k\ge 2$, we have
$S_k\subseteq (k-2,k-1)$ and then conclude that $\posf^A(k)=k$.  But this
example is a triviality and tells us nothing about the structure of addition
chains.  

So we ask:

\begin{quest}
\label{genomegakk}
Assuming that Conjecture~\ref{omegakk} holds, what conditions on $A$ are
needed to ensure that Conjecture~\ref{omegakk} holds when $\mathscr{D}^A$ is
used in place of $\mathscr{D}^\acl$ or $\mathscr{D}^*$?  Does it hold when $A$
is the set of Hansen chains,
or the set of chains which are star or quasi-star?
\end{quest}

\subsection*{Acknowledgements}

The author thanks J.~C.~Lagarias for suggesting this problem and
for helpful discussions and editorial advice. 
He thanks Andreas Blass for supplying references and J. Arias de Reyna
for suggestions improving the paper.
Work of the author was supported by NSF grants DMS-0943832 and DMS-1101373.

\end{document}